\documentclass[11pt,oneside,english]{amsart}
\usepackage[T1]{fontenc}
\usepackage[latin9]{inputenc}
\usepackage{babel}
\usepackage{amsbsy}
\usepackage{amsthm}
\usepackage{amssymb}
\usepackage{stmaryrd}
\usepackage[unicode=true,pdfusetitle,
 bookmarks=true,bookmarksnumbered=false,bookmarksopen=false,
 breaklinks=false,pdfborder={0 0 1},backref=false,colorlinks=false]
 {hyperref}
\hypersetup{
  colorlinks=true,  linkcolor=blue,  citecolor=red}

\makeatletter
\numberwithin{equation}{section}
\numberwithin{figure}{section}

\usepackage{amsfonts}
\usepackage{amsmath}
\usepackage{amsfonts}\usepackage{amsthm}\usepackage{color}\usepackage{enumitem}\usepackage{bm}\usepackage{cancel}\usepackage{thmtools}
\usepackage{hyperref}
\usepackage{tikz}
\usetikzlibrary{patterns}

\def\theenumi{\arabic{enumi}}

\def\theenumii{\alph{enumii}}
\def\p@enumii{\theenumi.}

\def\theenumiii{\arabic{enumiii}}
\def\p@enumiii{(\theenumi)(\theenumii)}

\def\p@enumiv{\p@enumiii.\theenumiii}
\newtheorem{theorem}{Theorem}[section]

\newtheorem{lemma}[theorem]{Lemma}

\newtheorem{proposition}[theorem]{Proposition}

\theoremstyle{definition}

\newtheorem{remark}[theorem]{Remark}

\def\dd{\mathrm{d}}
\makeatother

\begin{document}
\title{On the spectrum of asymptotic entropies of random walks}
\begin{abstract}  
Given a random walk on a free group, we study the random walks it
induces on the group's quotients. We show that the spectrum of asymptotic
entropies of the induced random walks has no isolated points, except
perhaps its maximum.
\end{abstract}

\author{Omer Tamuz}
\author{Tianyi Zheng}
\thanks{Omer Tamuz was supported by a grant from the Simons Foundation (\#419427).}
\address{Omer Tamuz, California Institute of Technology, Pasadena, CA 91125.
Email: omertamuz@gmail.com.}
\address{Tianyi Zheng, University of California, San Diego, La Jolla, CA 92093.
Email: tzheng2@math.ucsd.edu.}
\maketitle

\section{Introduction}

Let $G$ be a finitely generated group, and let $\mu$ be a probability
measure on $G$. The $\mu$-random walk on $G$ is a time homogeneous
Markov chain $g_{1},g_{2},\ldots$ on the state space $G$ whose steps
are distributed i.i.d.\,$\mu$: for $g,h\in G$ the transition probability
from $g$ to $h$ is $\mu(g^{-1}h)$. An important statistic of a
random walk is its \emph{Avez Asymptotic Entropy \cite{avez1974theorem}}
\[
h(G,\mu):=\lim_{n\to\infty}\frac{1}{n}H\left(g_{n}\right),
\]
where $H(\cdot)$ is the Shannon entropy. The importance of asymptotic
entropy is due to the fact that it vanishes if and only if every
bounded $\mu$-harmonic function is constant; that is, if the
$\mu$-random walk has a trivial Poisson boundary
\cite{avez1974theorem,kaimanovich1983random}. Moreover, as the
asymptotic entropy is the limit of the mutual information
$I(g_{1};g_{n})$ between the first step of the random walk and its
position in later time periods, it quantifies the extent by
which the random walk fails to have the Liouville property.

Suppose that $G$ has $d$ generators, and let $\mu$ be the symmetric
measure that assigns $1/(2d)$ to each generator and its inverse. The
main question that we ask in this paper is: what possible values of
$h(G,\mu)$ are attained as we vary the group $G$?

To formalize and generalize this question, we consider the following
setting.  Given $G$ and $\mu$, and given a quotient $\Gamma=G/N$, the
induced random walk $g_{1}N,g_{2}N,\ldots$ on $\Gamma$ has step
distribution $\mu_{\Gamma}$, where, for $\gamma=gN$,
$\mu_\Gamma(\gamma) = \mu(gN)$. In other words, $\mu_\Gamma$ is the
push-forward of $\mu$ under the quotient map; we will simply write
$\mu$ instead of $\mu_{\Gamma}$ whenever this is unambiguous. For a
given $G$ and $\mu$, what values can be realized as the asymptotic
random walk entropies of such quotients? This is particularly
interesting when $G$ has many quotients, and we indeed focus on the
case that of $\mathbf{F}_{d}$, the free groups with $d\geq2$
generators.

Given $(G,\mu)$ we denote the spectrum of random walk entropies by
\[
\mathfrak{H}(G,\mu):=\{h(\Gamma,\mu):\ \Gamma\mbox{ is a quotient group of }G\}.
\]
We will consider measures  $\mu$ on $G$ that have
\emph{finite first moment}, that is,
$\sum_{g\in G}\left|g\right|_{S}\mu(g)<\infty$, where
$\left|\cdot\right|_{S}$ is the word length with respect to generating
set $S$. Recall that $\mu$ is \emph{non-degenerate} if its support generates $G$
as a semigroup.

Our main result is the following.
\begin{theorem}\label{main}

Let $\mu$ be a non-degenerate probability measure with finite first
moment on the free group $\mathbf{F}_{d}$, $d\ge2$. Suppose $\Gamma$
is a proper quotient of ${\bf F}_{d}$. Then for any $\epsilon>0$,
there exists a quotient group $\tilde{\Gamma}$ of $\mathbf{F}_{d}$
such that $\mathbf{F}_{d}\twoheadrightarrow\tilde{\Gamma}\twoheadrightarrow\Gamma$
and 
\[
h(\Gamma,\mu)<h(\tilde{\Gamma},\mu)<h(\Gamma,\mu)+\epsilon.
\]
In particular, the set $\mathfrak{H}({\bf F}_{d},\mu)$
has no isolated points, except perhaps its maximum.

\end{theorem}

It follows from Theorem \ref{main} that if
$\mathfrak{H}(\mathbf{F}_{d},\mu)$ is a closed subset in $\mathbb{R}$,
then it must be the full
interval $\left[0,h({\bf F}_{d},\mu)\right]$.  To the
best of our knowledge, it is not known whether the set
$\mathfrak{H}(\mathbf{F}_{d},\mu)$ is closed.

The key ingredient in the proof of Theorem~\ref{main}
is an explicit construction, which might be of independent interest,
of a sequence of groups in the space $\mathcal{G}_{d}$ of
$d$-marked groups with the following properties.

\begin{proposition}\label{sequence}

Let $\mu$ be a non-degenerate probability measure on $\mathbf{F}_{d}$,
$d\ge2$, with finite first moment. Then there exists a sequence of
marked groups $\left((\Gamma_{n},S_{n})\right)_{n=1}^{\infty}$ in
$\mathcal{G}_{d}$ such that:
\begin{description}
\item [{(i)}] The sequence $(\Gamma_{n},S_{n})$ converges to $(\mathbf{F}_{d},{\bf S})$
as $n\to\infty$ in the space of $d$-marked groups.
\item [{(ii)}] The sequence of asymptotic entropies $h(\Gamma_{n},\mu)\to0$
as $k\to\infty$.
\item [{(iii)}] For each $n\in\mathbb{N}$, $\Gamma_{n}$ is non-amenable,
has no nontrivial amenable normal subgroups, and has only countably
many amenable subgroups.
\end{description}
\end{proposition}

The moment condition on $\mu$ is used to bound the asymptotic
entropy. It seems to be an interesting question whether Proposition
\ref{sequence} remains true assuming only that $\mu$ has finite
entropy.

Property (iii) in the statement above implies that the action of $\Gamma_{n}$
on the Poisson boundary of $(\Gamma_{n},\mu)$ is essentially free.
This property is crucial for our purposes. Any sequence of $d$-marked
finite groups with girth growing to infinity would satisfy properties
(i) and (ii), but the Poisson boundaries are trivial for finite groups. 

We construct the sequence of marked groups as stated via taking extensions
of the Fabrykowski-Gupta group. Necessary terminology and background are
reviewed in Section \ref{sec:prelim}. Provided the sequence of marked
groups stated in Proposition \ref{sequence}, the proof of Theorem
\ref{main} is completed by taking suitable diagonal product of groups;
see Section \ref{sec:joining}.

\subsection{Boundary entropies}
A closely related question---and, to our knowledge, a much better
studied one---is that of the spectrum of \emph{Furstenberg entropies.}
Let $(X,\nu)$ be a standard Borel space, equipped with a probability
measure, and on which $G$ acts by measure class preserving transformations.
We say that $(X,\nu)$ is a \emph{$(G,\mu)$-space}, if the measure
$\nu$ is $\mu$-stationary, that is $\mu\ast\nu=\nu$.\emph{ }The
Furstenberg entropy of a $(G,\mu)$-space $(X,\nu)$ is a numerical
invariant defined in \cite{furstenberg1963noncommuting} as

\[
h_{\mu}(X,\nu):=\sum_{g\in G}\mu(g)\int_{X}-\log\frac{\dd g^{-1}\nu}{\dd\nu}\dd\nu.
\]

The \emph{Furstenberg entropy realization problem} asks given $(G,\mu)$,
what is the spectrum of the Furstenberg entropy $h_{\mu}(X,\nu)$,
as $(X,\nu)$ varies over all ergodic $\mu$-stationary actions of
$G$.

We briefly summarize what is known about this problem. In Kaimonovich and
Vershik \cite{kaimanovich1983random} it is shown that $h_{\mu}(X,\nu)\leq h(G,\mu)$.
The Poisson boundary of an induced random walk on a
quotient group $G/N$ is a $(G,\mu)$-space, whose Furstenberg entropy is equal to
the random walk's asymptotic entropy. Hence every realizable random
walk entropy value is also a realizable Furstenberg entropy value.

Nevo \cite{Ne03} shows that whenever $G$ has Kazhdan's property (T) then there
is a constant $c>0$, depending on $(G,\mu)$, such that whenever $h_{\mu}(X,\nu)<c$ then
it in fact vanishes.  In \cite{bowen2010random}, Bowen showed that for
the free group $\mathbf{F}_{d}$, $d\ge2$, and $\mu$ uniform on the
symmetric free generating set ${\bf S}\cup{\bf S}^{-1}$, all values in
$[0,h({\bf F}_{d},\mu)]$ can be realized as the Furstenberg
entropy of an ergodic stationary action of ${\bf F}_{d}$.\footnote{The
approach in \cite{bowen2010random} is to take an ergodic invariant
random subgroup of $G$ and construct an ergodic stationary system
(which can be referred to as a Poisson bundle, using the terminology
introduced in \cite{kaimanovich2005amenability}). The Furstenberg
entropy of this stationary system is then studied by considering
random walk entropies on the coset spaces associated with the
invariant random subgroups.  Recall that an IRS is a Borel probability
measure $\eta$ on the Chabauty space ${\rm Sub}(G)$ of closed subgroups of
$G$, which is invariant under conjugation by $G$.  For further work on
the Furstenberg entropy realization problem using the IRS-Poisson
bundle approach, see
\cite{hartman2015furstenberg,hartman2018furstenberg} and references
therein.}

A particularly important class of $(G,\mu)$-space are the
$(G,\mu)$-\emph{boundaries}. These are the $G$-factors of the Poisson
boundary of $(G,\mu)$, and include the Poisson boundaries of the
induced random walks on quotient groups. For such boundaries, the next result is an analogue of Theorem~\ref{main}.

\begin{theorem}\label{B}

In the setting of Theorem \ref{main}, suppose $(X,\nu)$ is a $\left(\mathbf{F}_{d},\mu\right)$-boundary
such that the action of ${\bf F}_{d}$ is not essentially free. Then
for any $\epsilon>0$, there exists a $\left(\mathbf{F}_{d},\mu\right)$-boundary
$\left(\tilde{X},\tilde{\nu}\right)$ such that 
\[
h(X,\nu)<h(\tilde{X},\tilde{\nu})<h(X,\nu)+\epsilon,
\]
and $(X,\nu)$ is an ${\bf F}_{d}$-factor of $(\tilde{X},\tilde{\nu})$.

\end{theorem}

Note that if an ergodic invariant random
subgroup is not almost surely a normal subgroup, then the corresponding
Poisson bundle is not a quotient of the Poisson boundary of $(G,\mu)$
because of the measure-preserving factor to the invariant random
subgroup.  Hence Bowen's results do not resolve the question for
Furstenberg entropies of $({\bf F}_{d},\mu)$-boundaries, or for
asymptotic random walk entropies.

\subsection{Spectral radii}

The same kind of construction as in the proof of Theorem \ref{main} implies the following result on spectral
radii of symmetric random walks. Recall that the spectral radius of
a $\mu$-random walk on $\Gamma$ is defined as 
\[
\rho(\Gamma,\mu)=\limsup_{2n\to\infty}\mu^{(n)}(id_{\Gamma})^{\frac{1}{2n}},
\]
where $\mu^{(n)}$ is the $n$-fold convolution of $\mu$ with itself.

\begin{theorem}\label{spec}

Let $\mu$ be a symmetric non-degenerate probability measure on the
free group $\mathbf{F}_{d}$, $d\ge2$. Suppose $\Gamma$ is a proper
quotient of ${\bf F}_{d}$. Then for any $\epsilon>0$, there
exists a quotient group $\tilde{\Gamma}$ of $\mathbf{F}_{d}$ such
that $\mathbf{F}_{d}\twoheadrightarrow\tilde{\Gamma}\twoheadrightarrow\Gamma$
and 
\[
\rho(\Gamma,\mu)-\epsilon<\rho(\tilde{\Gamma},\mu)<\rho(\Gamma,\mu).
\]

\end{theorem}

Our construction uses a diagonal product of marked groups, and  is similar to the construction
in \cite{KP}. A result of Kassabov and Pak \cite{KP2} states that
the set of the spectral radii $\{\rho(\Gamma,\mu):\Gamma\mbox{ is a quotient of }{\bf F}_{d}\}$
contains a subset homeomorphic to the Cantor set. The same construction
shows that the set $\mathfrak{H}({\bf F}_{d},\mu)$ contains
a subset homeomorphic to the Cantor set as well. It is not known whether
this set of spectral radii is closed.

\subsection*{Acknowledgement}

We thank Michael Bj\"orklund, J\'er\'emie Brieussel, Yair Hartman
and Igor Pak for helpful discussions. 

\section{Preliminaries\label{sec:prelim}}

\subsection{$(G,\mu)$-boundaries\label{subsec:boundaries}}

In this paper we only consider countable groups. A probability measure
$\mu$ on $G$ is non-degenerate if the support of $\mu$ generates $G$
as a semigroup. For a countable group $G$, we say a Lesbesgue space
$(X,\nu)$ is a $G$-space, if $G$ acts measurably on $X$ and the
probability measure $\nu$ is quasi-invariant with respect to the
$G$-action. A $G$-space $(X,\nu)$ is ergodic if every $G$-invariant
subset is either null or conull. A measurable map $\pi:(X,\nu)\to(Y,\eta)$
is called a $G$-\emph{map }if it is $G$-equivariant and $\eta$
is the pushforward of $\nu$ under $\pi$.

Given a probability measure $\mu$ on $G$, let $\Omega=G^{\mathbb{N}}$
be the path space, $\mathbb{P_{\mu}}$ be the law of the $\mu$-random
walk starting at $id$, and $\mathcal{I}$ be the $\sigma$-field
on $\Omega$ that is invariant under time shifts. The Poisson boundary
of $(G,\mu)$ is denoted by the measure space $(B,\mathcal{F},\nu_{B})$
together with a $G$-map ${\bf b}:\left(\Omega,\mathcal{I},\mathbb{P}_{\mu}\right)\to\left(B,\mathcal{F},\nu_{B}\right)$,
where ${\bf b}^{-1}\mathcal{F}=\mathcal{I}$ up to null sets with
respect to $\mathbb{P}_{\mu}$, and the $\sigma$-algebra $\mathcal{F}$
is countably generated and separating points. The existence and uniqueness
up to isomorphism of the Poisson boundary of $(G,\mu)$ was shown
by Furstenberg \cite{furstenberg1963noncommuting,furstenberg1971random,furstenberg1973boundary}.
The $G$-action on the Poisson boundary $(B,\nu_{B})$ is ergodic,
and in fact doubly ergodic, by Kaimanovich \cite{kaimanovich2003double}.

We use the notation $(B,\nu_{B})$ to denote a compact model of the
Poisson boundary of $(G,\mu)$, which exists by the Mackey realization
\cite{mackey1962point}. A $(G,\mu)$-boundary $(X,\nu)$ is defined
to be a $G$-factor of $(B,\nu_{B})$. Moreover, the factor map $\left(B,\nu_{B}\right)\to(X,\nu)$
is essentially unique, see \cite[Theorem 2.14]{badershalom}, and
we will denote it by $\boldsymbol{\beta}_{X}$.

Denote by $P(X)$ the space of Borel probability measures on the compact
space $X$. A factor map $\pi:(Y,\eta)\to(X,\nu)$ gives a unique
disintegration map $D_{\pi}:X\to P(Y)$ such that for $\nu$-a.e.
$x\in X$, $D_{\pi}(x)$ is supported on the fiber of $x$ and $\int_{X}D_{\pi}(x)d\nu(x)=\eta$.
We say $(Y,\eta)$ is a \emph{relatively measure preserving extension}
of $X$ if $D_{\pi}$ is $G$-equivariant, that is $D_{\pi}(g\cdot x)=g\cdot D_{\pi}(x)$.

We will need the following properties regarding Furstenberg entropy
and relatively measure preserving extensions.

\begin{proposition}[{\cite[Proposition 1.9]{nevo2000rigidity}}]\label{eq}

Let $\pi:(Y,\eta)\to(X,\nu)$ be a $G$-factor map. Suppose $h(X,\nu)<\infty$
and $h(Y,\eta)=h(X,\nu)$. Then $(Y,\eta)$ is a relative measure
preserving extension of $(X,\nu)$.

\end{proposition}

\begin{lemma}[{\cite[Corollary 2.20]{badershalom}}]\label{mp}

Let $\pi:(Y,\eta)\to(X,\nu)$ be a relatively measure-preserving extension
of two $(G,\mu)$-boundaries. Then $(Y,\eta)=(X,\nu)$.

\end{lemma}

\subsection{The space of marked groups and convergence to the free group\label{subsec:chabauty}}

Denote by $\mathcal{G}_{d}$ the space of $d$-generated groups $(G,S)$,
where $S=(s_{1},\ldots,s_{d})$ is a generating tuple, equipped with
the \emph{Cayley-Grigorchuk topology}. 
We refer to the pair $(G,S)$
as a marked group and $\mathcal{G}_{d}$ the space of $d$-marked
groups. Recall that in this topology, two marked groups $(G_{1},S_{1})$
and $(G_{2},S_{2})$ are close if marked balls of large radius in
the Cayley graphs of $(G_{1},S_{1})$ and $(G_{2},S_{2})$ around
the identities are isomorphic. This space is introduced by Grigorchuk in \cite{grigorchuk1984}. 

Denote by $({\bf F}_{d},{\bf S})$ a free group of rank $d$, where
${\bf S}=({\bf s}_{1},\ldots,{\bf s}_{d})$ consists of the free generators.
Let $G$ be a $d$-generated group. Following the definition in Akhmedov
\cite{akhmedov} and Olshanskii-Sapir \cite{olshanskii2009freelike},
we say a non-trivial word $w(x_{1},\ldots,x_{d})$ is a \emph{$d$-almost-identity}
for $G$, if the identity $w(g_{1},\ldots,g_{d})=1$ is satisfied
for any $d$-generating tuple $(g_{1},\ldots,g_{d})$. By \cite[Theorem 9]{olshanskii2009freelike},
there exists a sequence of $d$-markings $(G,S_{k})_{k=1}^{\infty}$
that converges to $({\bf F}_{d},{\bf S})$ in the Cayley-Grigorchuk
topology if $G$ is $d$-generated and satisfies no $d$-almost identity. 

In \cite{abert2005}, Ab\'ert gives a general criterion for a group
to satisfy no identity. Suppose $G\curvearrowright X$ by permutations.
We say $G$ {\it separates} $X$, if for every finite subset $Y$ of $X$,
the pointwise fixator $G_{Y}=\{g\in G:y\cdot g=y\mbox{ for all }y\in Y\}$
has no fixed point outside $Y$. Ab\'ert shows that if $G$ separates
$X$ then $G$ satisfies no identity. Bartholdi and Erschler \cite{bartholdi2015order}
provide a criterion for absence of almost-identities: under the additional
assumption that the Frattini subgroup $\Phi(G)$ has finite index
in $G$, the condition in Ab\'ert's criterion implies that $G$ satisfies
no almost-identity. Recall that the Frattini subgroup of $G$ is the
intersection of all the maximal subgroups of $G$.

Weakly branch groups provide examples of groups satisfying Ab\'ert's
criterion. The notion of weakly branch group is introduced by Grigorchuk
in \cite{grigorchuk2000}. Let $\mathsf{T}$ be a rooted spherical
symmetric tree. For a vertex $u\in\mathsf{T}$, let $C_{u}$ be the
set of infinite rays with prefix $u$. We say a group $G$ acting
by automorphisms on $\mathsf{T}$ is {\it weakly branching} if it acts level transitively and 
the rigid stabilizer ${\rm Rist}_{G}(C_{u})$ of any vertex $u\in\mathsf{T}$ is nontrivial. 
Recall that ${\rm Rist}_{G}(C_{u})=\{g\in G:x\cdot g=x\mbox{ for all }x\notin C_{u}\}$, that is, the set of 
group elements that only move the descendants of $u$.
If $G$ is weakly branching, then $G$
separates the boundary $\partial T$ of the tree, see \cite[Proof of Corollary 1.4]{abert2005}. If in addition, 
the product of rigid stabilizers $\prod_{u\in \mathsf{T}_n} {\rm Rist}_{G}(C_{u})$ is a finite index subgroup of $G$ for every $n$,
then $G$ is said to be a {\it branch group}.

\section{A sequence of marked groups \label{sec:sequence}}

This section is devoted to the proof of Proposition \ref{sequence}.
To fix ideas, we start with the Fabrykowski-Gupta group introduced
in \cite{fabrykowskigupta}. It is a group
acting on the ternary rooted tree $\mathsf{T}$. Encode vertices of $\mathsf{T}$
by finite strings in the alphabet ${0,1,2}$, and the boundary of the
tree by infinite strings in ${0,1,2}$. Denote by $\mathsf{T}_{n}$ the level $n$ vertices of the
rooted tree $\mathsf{T}$ and ${\rm St}_{G}(n)$ the level
$n$ stabilizer, that is, 
${\rm St}_{G}(n)=\{g\in G: u.g=u \text{ for all }u\in\mathsf{T}_{n}\}$.

The Fabrykowski-Gupta group is generated by two elements:
a root permutation $a$ which permutes the three subtrees of the
root cyclically and a directed permutation $b$ which fixes the right
most ray $2^{\infty}$ and is defined recursively by 
\[
b=(a,id,b).
\]
In other words we have for any ray $w\in\{0,1,2\}^{\infty}$,
\begin{align*}
0w\cdot a & =1w,\ 1w\cdot a=2w,\ 2w\cdot a=0w;\\
0w\cdot b & =0(w\cdot a),\ 1w\cdot b=1w,\ 2w\cdot b=2(w\cdot b).
\end{align*}
See Figure~\ref{fig:schreier}. For more background on groups acting on
rooted trees and the notation of wreath recursion see the reference
\cite{handbook}. The group $G=\left\langle a,b\right\rangle $ is
called the Fabrykowski-Gupta group. It is an example of non-torsion
Grigorchuk-Gupta-Sidki (GGS) groups.

\begin{figure}
  \begin{center}
    \includegraphics[scale=0.7]{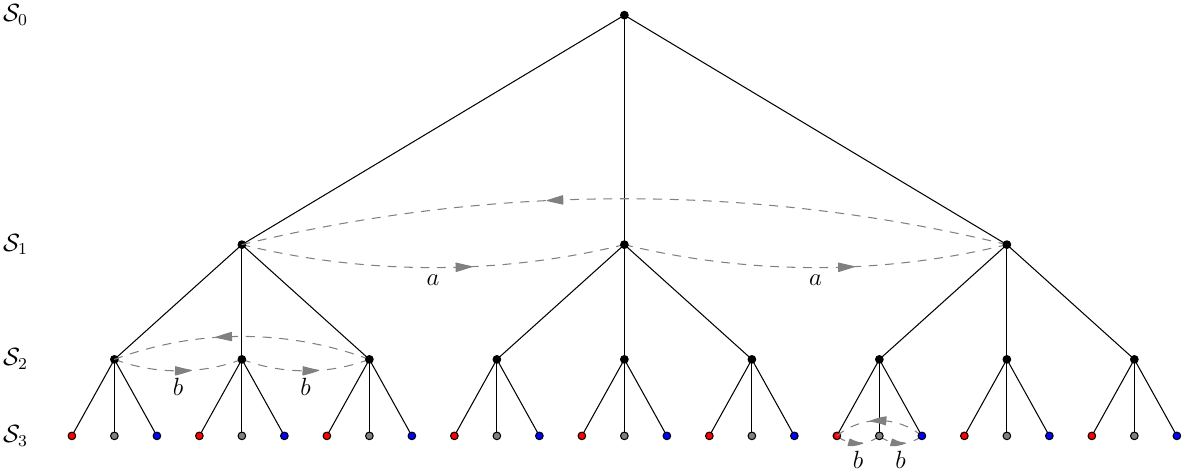}
    
    \caption{\label{fig:schreier}The action of the Fabrykowski-Gupta
      group on the first four levels of the rooted ternary tree. Self
      loops are not depicted. Arrows show the action on the roots of
      subtrees, with corresponding arrows in the rest of the subtree
      not drawn explicitly. The restriction of $\mathcal{S}_3$ to its
      red (and likewise gray and blue) vertices forms a copy of
      $\mathcal{S}_2$. The Schreier graph $\mathcal{S}_3$ is the
      disjoint union of these three graphs, with the addition of the
      three edges labeled by $b$.}
  \end{center}
\end{figure}

The group $G=\left\langle a,b\right\rangle $ is known to have the
following properties:
\begin{itemize}
\item (\cite{bg02}) $G$ is a just infinite branch group which is regularly
branching over its commutator group $[G,G]$.
\item (\cite{fabrykowskiguptaII,bartholdipochon}) $G$ is of intermediate
growth.
\item (\cite{FGU}) $G$ has the congruence subgroup property: every finite
index subgroup of $G$ contains some level stabilizer ${\rm St}_{G}(n)$. 
\end{itemize}

\subsection{Permutation wreath extensions}

Let $G_{n}$ be the quotient group $G/{\rm St}_{G}(n)$,
which acts faithfully and transitively on $\mathsf{T}_{n}$. We
denote by $\bar{a},\bar{b} \in G_n$ the images of the generators $a,b$ under the
quotient map $G \to  G/{\rm St}_{G}(n)$. Consider
the level $n$ Schreier graph $\mathcal{S}_{n}$ with vertex set $\mathsf{T}_{n}$
and edge set $E=\{(x,x\cdot \bar{a}),(x,x\cdot \bar{b}):\ x\in\mathsf{T}_{n}\}$.
It is a finite graph on $3^{n}$ vertices. Consider the permutation
wreath product of the free product ${\bf A}=(\mathbb{Z}/3\mathbb{Z})\ast(\mathbb{Z}/3\mathbb{Z})$
and $G_{n}$ over the set $\mathsf{T}_{n}$, that is, 
\[
\mathbf{A}\wr_{\mathsf{T}_{n}}G_{n}=\left(\oplus_{\mathsf{T}_{n}}{\bf A}\right)\rtimes G_{n},
\]
where $G_{n}$ acts on $\oplus_{\mathsf{T}_{n}}{\bf A}$ by permuting
the coordinates. We write elements of $\mathbf{A}\wr_{\mathsf{T}_{n}}G_{n}$
as pairs $\left(\varphi,g\right)$, where $\varphi\in\oplus_{\mathsf{T}_{n}}{\bf A}$
is regarded as a function $\mathsf{T}_{n}\to\mathbf{A}$ and $g\in G_{n}$. 

Denote by $s$ and $t$ the two standard generators of $\mathbf{A}$,
${\bf A}=\left\langle s,t|s^{3}=t^{3}=1\right\rangle $. Consider
the subgroup $W_{n}$ of $\mathbf{A}\wr_{\mathsf{T}_{n}}G_{n}$ generated
by 
\begin{equation}
a_{n}=\left(id,\bar{a}\right),\ b_{n}=\left(\delta_{2^{n-1}0}^{s}+\delta_{2^{n-1}1}^{id}+\delta_{2^{n}}^{t},\bar{b}\right),\label{eq:Gamman}
\end{equation}
where in the direct sum $\oplus_{\mathsf{T}_{n}}{\bf A}$, $\delta_{x}^{\gamma}$
denotes the function that is $\gamma$ at $x$ and identity elsewhere.
We use additive notation $\delta_{x}^{\gamma_{1}}+\delta_{y}^{\gamma_{2}}$,
$x\neq y$, for the function that is $\gamma_{1}$ at $x$, $\gamma_{2}$
at $y$, and identity elsewhere. 

The choice of $(a_{n},b_{n})$ guarantees:

\begin{lemma}\label{conv}

The sequence $\left(W_{n},(a_{n},b_{n})\right)$ converges to $(G,(a,b))$
in the Cayley-Grigorchuk topology as $n\to\infty$. Indeed, $(W_{n+1},(a_{n+1},b_{n+1}))$
is a marked quotient of $\left(W_{n},(a_{n},b_{n})\right)$, and the
ball of radius $2^{n-2}$ around $id$ in the Cayley graph of $(W_{n},(a_{n},b_{n}))$
coincide with the ball of radius $2^{n-2}$ around $id$ in $(G,(a,b))$. 

\end{lemma}

\begin{proof}

The Fabrykowski-Gupta group belongs to the class of bounded automaton
groups. Schreier graphs of bounded automaton groups are studied systematically
in Bondarenko's dissertation \cite{bondarenko2007groups}. In particular,
on the finite Schreier graph $\mathcal{S}_{n}$, we have that the
graph distance between the vertices $2^{n-1}0,2^{n}$ satisfy $d_{\mathcal{S}_{n}}(2^{n-1}0,2^{n})=2^{n}-1$.
For more details see \cite[Chapter VI]{bondarenko2007groups}. 

Note that $G$ embeds as a subgroup of $G\wr_{\mathsf{T}_{n}}G_{n}$,
where the embedding is given by the wreath recursion 
\[
a\mapsto(id,\bar{a}),\ b\mapsto\left(\delta_{2^{n-1}0}^{a}+\delta_{2^{n-1}1}^{id}+\delta_{2^{n}}^{b},\bar{b}\right).
\]
Now consider a word $w=w_{1}\ldots w_{\ell}$, where $w_{j}\in\{a^{\pm1},b^{\pm1}\}$
and evaluate this word in $G\wr_{\mathsf{T}_{n}}G_{n}$ by the embedding
above. Denote the image in $G\wr_{\mathsf{T}_{n}}G_{n}$ by $(\phi_w,\bar{w})$. 
For the configuration $\phi_{w}\in\oplus_{\mathsf{T}_{n}}G$,
we have that 
\[
\phi_{w}(x)=\prod_{i=1}^{n}\phi_{w_{i}}(x\cdot w_{1}\ldots w_{i-1}).
\]
It follows from the triangle inequality that if $\ell\le2^{n-2}$,
then the trajectory $\{x,x\cdot w_{1},\ldots,x\cdot w_{1}\ldots w_{\ell-1}\}$
can visit at most one point in the set $\{2^{n-1}0,2^{n}\}$. In particular,
$\phi_{w}(x)$ is an element in either $\left\langle a\right\rangle $
or $\left\langle b\right\rangle $. Thus if we evaluate
the same word $w$ in $W_n$ under $a\mapsto a_{n}$ and $b\mapsto b_{n}$,
the resulting element $\left(\tilde{\phi}_{w},\bar{w}\right)$ can
be identified with $\left(\phi_{w},\bar{w}\right)$ in $G\wr_{\mathsf{T}_{n}}G_{n}$.
Namely, $\phi_{w}$ is obtained from $\tilde{\phi}_{w}$ by replacing
$s$ with $a$ and $t$ with $b$ and vice versa. 

The quotient map from $\left(W_{n},(a_{n},b_{n})\right)$ to $\left(W_{n+1},(a_{n+1},b_{n+1})\right)$
is given as follows. Note that $\mathbf{A}\wr_{\mathsf{T}_{n+1}}G_{n+1}=\left(\mathbf{A}\wr_{\{0,1,2\}}\left\langle a\right\rangle \right)\wr_{\mathsf{T}_{n}}G_{n}$.
Let $\tau:\mathbf{A}\to\mathbf{A}\wr_{\{0,1,2\}}\left\langle a\right\rangle $
be the group homomorphism determined by $\tau(s)=(id,a)$ and $\tau(t)=\left(\delta_{0}^{s}+\delta_{1}^{id}+\delta_{2}^{t},id\right)$.
The homomorphism $\tau$ extends to $\oplus_{\mathsf{T}_{n}}{\bf A}\to\oplus_{\mathsf{T}_{n}}\left(\mathbf{A}\wr_{\{0,1,2\}}\left\langle a\right\rangle \right)$
coordinate-wise, that is $\tau(\phi)(x)=\tau(\phi(x))$, $x\in\mathsf{T}_{n}$.
It follows that the wreath recursion formula in $G$ that the
map 
\begin{align*}
W_{n} & \to W_{n+1}\\
(\phi,g) & \mapsto(\tau(\phi),g)
\end{align*}
is a marked group epimorphism which sends $a_{n}$ to $a_{n+1}$ and $b_{n}$
to $b_{n+1}$. 

\end{proof}

Next we verify that $W_{n}$ is virtually a direct product of free groups:

\begin{lemma}\label{virtual}

The group $W_{n}$ contains $\oplus_{\mathsf{T}_{n}}[{\bf A},{\bf A}]$
as a finite index normal subgroup.

\end{lemma}

\begin{proof}

We proceed by induction on $n$.

As in the proof of Lemma \ref{conv}, let $\tau:\mathbf{A}\to\mathbf{A}\wr_{\{0,1,2\}}\left\langle a\right\rangle $
be the group homomorphism determined by $\tau(s)=(id,a)$ and $\tau(t)=\left(\delta_{0}^{s}+\delta_{1}^{id}+\delta_{2}^{t},id\right)$,
where $a$ is the $3$-cycle $(0,1,2)$. When $n=1$, by definition
$W_{1}$ is generated by $a_{1}=\tau(s)$ and $b_{1}=\tau(t)$. Since
$a_{1}^{-1}b_{1}a_{1}=(\delta_{0}^{id}+\delta_{1}^{t}+\delta_{2}^{s},id)$,
it follows that the projection of $W_{1}\cap\oplus_{\mathsf{T}_{1}}{\bf A}$
to the component over vertex $2$ is ${\bf A}$. Direct calculation
shows that $\left[b_{1}a_{1}^{-1}b_{1}a_{1},a_{1}b_{1}a_{1}^{-1}b_{1}\right]=\left(\delta_{2}^{sts^{-1}t^{-1}},id\right).$
 It follows that $\left[W_{1},W_{1}\right]\cap\oplus_{\mathsf{T}_{1}}{\bf A}$
contains $\left\{ \left(\delta_{2}^{\gamma},id\right):\ \gamma\in\left\langle sts^{-1}t^{-1}\right\rangle ^{\mathbf{A}}\right\} $,
while the normal closure $\left\langle sts^{-1}t^{-1}\right\rangle ^{\mathbf{A}}$
is exactly the commutator subgroup $[{\bf A},{\bf A}]$. Since $\tau(s)$
acts as a $3$-cycle permuting $\mathsf{T}_{1}=\{0,1,2\}$, it follows
that $\left[W_{1},W_{1}\right]\cap\oplus_{\mathsf{T}_{1}}{\bf A}>\oplus_{\mathsf{T}_{1}}[{\bf A},{\bf A}]$.
The quotient group $W_{1}/\oplus_{\mathsf{T}_{1}}[{\bf A},{\bf A}]$
is a subgroup of $({\bf A}/[{\bf A},{\bf A}])\wr_{\mathsf{T}_{1}}\left\langle a\right\rangle $,
which is finite. 

We have shown that $\tau([{\bf A},{\bf A}])$ contains $\oplus_{\mathsf{T}_{1}}[{\bf A},{\bf A}]$
as finite index normal subgroup, which reflects the property that
$G$ is regularly branching over its commutator subgroup. 

Suppose the statement is true for $n$ that $W_{n}$ contains $\oplus_{\mathsf{T}_{n}}[{\bf A},{\bf A}]$
as a finite index normal subgroup. To prove the claim for $n+1$,
it suffices to show that $\left(\delta_{2^{n+1}}^{\gamma},id\right)\in
W_{n+1}$
for any $\gamma\in[{\bf A},{\bf A}]$. Recall the quotient map $\pi:W_{n}\to W_{n+1}$
explained in the proof of Lemma \ref{conv}, where $\mathbf{A}\wr_{\mathsf{T}_{n+1}}G_{n+1}$
is identified with $\left(\mathbf{A}\wr_{\{0,1,2\}}\left\langle a\right\rangle \right)\wr_{\mathsf{T}_{n}}G_{n}$.
By the induction hypothesis, $(\delta_{2^{n}}^{\sigma},id)\in W_{n}$
for any $\sigma\in[{\bf A},{\bf A}]$. Under the quotient map $\pi$,
we have 
\[
\pi\left((\delta_{2^{n}}^{\sigma},id)\right)=\left(\delta_{2^{n}}^{\tau(\sigma)},id\right).
\]
With the map $\tau$ we are back in the situation of the induction
base, where we have shown that $\tau\left(\left[{\bf A},{\bf A}\right]\right)$
contains $\oplus_{\mathsf{T}_{1}}[{\bf A},{\bf A}]$. In particular,
for any $\gamma\in[{\bf A},{\bf A}]$, there is an element $\sigma\in[{\bf A},{\bf A}]$
such that $\tau(\sigma)=(\delta_{2}^{\gamma},id)$. It follows that
$\pi\left((\delta_{2^{n}}^{\sigma},id)\right)=\left(\delta_{2^{n}}^{\tau(\sigma)},id\right)=\left(\delta_{2^{n+1}}^{\gamma},id\right)$,
in particular it is an element of $W_{n+1}$. 

\end{proof}

\subsection{Choices of marked subgroups}
\label{subsec:Choices-of-marked}

We are given a fixed rank $d\in\mathbb{N}$, $d\ge2$. 
Since the Fabrykowski-Gupta
group $G=\left\langle a,b\right\rangle $ is a branch group acting
faithfully on the ternary tree $\mathsf{T}$, by Ab\'ert's criterion
\cite[Theorem 1.1]{abert2005} and its proof, we have that given any $n\in\mathbb{N}$,
and any vertex $v\in\mathsf{T}$, there exist elements $\gamma_{1}^{(n)},\ldots,\gamma_{d}^{(n)}\in{\rm Rist}_{G}(v)$
such that $w\left(a\gamma_{1}^{(n)},b\gamma_{2}^{(n)},\gamma_{3}^{(n)},\ldots,\gamma_{d}^{(n)}\right)\neq id$
for all reduced word $w$ of length $1\le|w|\le n$. 

In what follows we fix the choice of $v$ to be the child of the root
indexed by $1$. For each $n\in\mathbb{N}$, fix a choice of $\gamma_{1}^{(n)},\ldots,\gamma_{d}^{(n)}\in{\rm Rist}_{G}(1)$
such that the tuple $\left(a\gamma_{1}^{(n)},b\gamma_{2}^{(n)},\gamma_{3}^{(n)},\ldots,\gamma_{d}^{(n)}\right)$
do not satisfy any reduced word $w$ of length $|w|\in\left[1,n\right]$. 

\begin{lemma}\label{transitive}

Denote by $H_{n}$ the subgroup of $G$ generated by the first two elements of the tuple chosen
above, 
\[
H_{n}=\left\langle a\gamma_{1}^{(n)},b\gamma_{2}^{(n)}\right\rangle.
\]
Then $H_{n}$ acts level transitively on the rooted ternary tree $\mathsf{T}$. 

\end{lemma}

\begin{proof}
The statement is equivalent to that the Schreier graph on
level $k$ vertices $\mathsf{T}_{k}$ with respect to $\left(a\gamma_{1}^{(n)},b\gamma_{2}^{(n)}\right)$
is connected. 

With respect to the original generators $a,b$, the Schreier graph
$\mathcal{S}_{k}$ with vertex set $\mathsf{T}_{k}$ under the action of
$G$ can be drawn recursively as follows, see \cite[Chapter
  V]{bondarenko2007groups}.  The level $1$ graph $\mathcal{S}_{1}$ on
$\{0,1,2\}$ has (directed) edges $(0,1)$, $(1,2)$, $(2,0)$ labeled by
$a$, and self loops at each vertex labeled by $b$. To draw
$\mathcal{S}_{k+1}$, take three copies of $\mathcal{S}_{k}$, append
letter $0$, $1$, or $2$ to the strings indexing the vertices of
$\mathcal{S}_{k}$ respectively in each copy. Then connect the three
copies by edges $(2^{k-1}00,2^{k-1}01)$, $(2^{k-1}01,2^{k-1}02)$ and
$(2^{k-1}02,2^{k-1}00)$ labeled by $b$. In Figure~\ref{fig:schreier}
this is depicted for $k=2$. There, $\mathcal{S}_3$ is seen to be the
union of three copies of $\mathcal{S}_2$, which are shown with red,
gray and blue vertices, respectively. The three additional edges
connecting them and labeled by $b$ are those in the bottom right of
the figure.

It follows that in $\mathcal{S}_{k}$, if we remove all vertices of
the form $1u$, $u\in\{0,1,2\}^{k-1}$ and edges connecting to such
vertices, the remaining graph is connected. Denote the remaining graph
by $\mathcal{S}_{k}'$. Since $\gamma_{1}^{(n)},\gamma_{2}^{(n)}$
are chosen to be in the rigid stabilizer of vertex $1$, in the subgraph
$\mathcal{S}_{k}'$, we may replace label $a$ by $a\gamma_{1}^{(n)}$
and label $b$ by $b\gamma_{2}^{(n)}$. The element $a\gamma_{1}^{(n)}$
moves $1u$, $u\in\{0,1,2\}^{k-1}$ into the vertex set of $\mathcal{S}_{k}'$,
namely, a string starting with $2$. It follows that the graph on
$\mathsf{T}_{k}$ with respect to $\left(a\gamma_{1}^{(n)},b\gamma_{2}^{(n)}\right)$
is connected. 
\end{proof}

Denote by $\ell_{n}$ the maximal word length of elements in the tuple
chosen, with respect to the original generating set $(a,b)$, that
is, 
\begin{equation}
\ell_{n}=\max\left\{ \left|a\gamma_{1}^{(n)}\right|_{\{a,b\}},\left|b\gamma_{2}^{(n)}\right|_{\{a,b\}},\ldots,\left|\gamma_{d}^{(n)}\right|_{\{a,b\}}\right\} .\label{eq:ln}
\end{equation}

Recall the subgroup $W_{k}$ of $\mathbf{A}\wr_{\mathsf{T}_{k}}G_{k}$,
generated by $\left(a_{k},b_{k}\right)$, as defined in (\ref{eq:Gamman}).
By Lemma \ref{conv}, the ball of radius $2^{k-2}$ around
the identity element in the Cayley graph of $(W_{k},(a_{k},b_{k}))$ coincide with
the ball of radius $2^{k-2}$ around the identity in $(G,(a,b))$. 
In particular, for $k\ge 2+\log_{2}(n\ell_{n})$, 
elements in the tuple $\left(a\gamma_{1}^{(n)},b\gamma_{2}^{(n)},\gamma_{3}^{(n)},\ldots,\gamma_{d}^{(n)}\right)$
have images in $\left(W_{k},(a_{k},b_{k})\right)$ under the identification
of balls of radius $2\ell_{n}$ around the identities. Record the
image tuple of elements in $W_{k}$ as $\left(h_{1}^{(n,k)},\ldots,h_{d}^{(n,k)}\right)$.
Finally, denote by $\Gamma_{n,k}$ the subgroup of
$W_{k}$ generated by the tuple $S_{n,k}=\left(h_{1}^{(n,k)},\ldots,h_{d}^{(n,k)}\right)$.
Note that by our choices, for $k\ge 2+\log_{2}(n\ell_{n}$, the ball of radius $n$ around the identity
in the Cayley graph of $\left(\Gamma_{n,k},S_{n,k}\right)$ is the
same as the ball of radius $n$ in the Cayley graph of the free group
$\left(\mathbf{F}_{d},\mathbf{S}\right)$. 

\subsection{Random walks on $\Gamma_{n,k}$}

Let $\mu$ be a non-denegerate probability measure on the free group
$\mathbf{F}_{d}$. Next we consider the $\mu$-random walk on the
group $\Gamma_{n,k}$ defined in the previous subsection, with $n\ge1$,
$k\gg\log_{2}(n\ell_{n})$. Our goal is to show that the action of $\Gamma_{n,k}$
on the Poisson boundary of $\left(\Gamma_{n,k},\mu\right)$ is essentially
free; and given any $\epsilon>0$, for all $k$ sufficiently large,
the asymptotic entropy of the $\mu$-random walk on $\Gamma_{n,k}$
is smaller than $\epsilon$. 

By \cite[Theorem 5.1]{BHT2017}, for a non-degenerate probability
measure $\mu$ on a countable group $\Gamma$, a sufficient condition
for the action of $\Gamma$ on the Poisson boundary of $(\Gamma,\mu)$
to be essentially free is that 
\begin{description}
\item [{(1)}] $\Gamma$ has only countably many amenable subgroups,
\item [{(2)}] $\Gamma$ does not contain any non-trivial normal amenable
subgroup, in other words, the amenable radical of $\Gamma$ is trivial.
\end{description}
We verify that these two properties are satisfied by $\Gamma_{n,k}$
in the following lemma.

\begin{lemma}\label{free}

For each $n$ and $k\ge2+\log_{2}(n\ell_{n})$, the group $\Gamma_{n,k}$
is non-amenable, has no non-trivial normal amenable subgroup, and
has only countably many amenable subgroups. 

\end{lemma}

\begin{proof}

We first introduce some notations. Given a subgroup $H$ of $\mathbf{A}\wr_{\mathsf{T}_{k}}G_{k}=\left(\oplus_{\mathsf{T}_{k}}{\bf A}\right)\rtimes G_{k}$,
for $g\in G_{k}$, let
\[
S_{H}(g):=\left\{ \phi\in\oplus_{\mathsf{T}_{k}}{\bf A}:\ (\phi,g)\in H\right\}.
\]
Then $S_{H}(id_{G_{k}})$ is a subgroup of $\oplus_{\mathsf{T}_{k}}{\bf A}$.
For $g\in G_k$, $S_{H}(g)$ is either empty, or a right coset of $S_{H}(id_{G_{k}})$ in
$\oplus_{\mathsf{T}_{k}}{\bf A}$. Denote by $\pi_{k}$ the natural
projection $\mathbf{A}\wr_{\mathsf{T}_{k}}G_{k}\to G_{k}$. 

As shown in the proof of Lemma \ref{transitive}, in the Schreier
graph on vertex set $\mathsf{T}_{k}$ with respect to the generating
tuple of $H_{n}$, there is a path connecting the vertex $2^{k}$
to the vertex $2^{k-1}0$, which does not visit any vertex starting
with letter $1$. It then follows by the definition of the generating
tuple $S_{n,k}=\left(h_{1}^{(n,k)},\ldots,h_{d}^{(n,k)}\right)$ of
$\Gamma_{n,k}$ that the set of values at the coordinate index by $2^{k}$ is all $\bf A$, 
that is,
$$\{\phi(2^k):\phi \in\cup_{g\in G_{k}}S_{\Gamma_{n,k}}(g)\}=\bf A$$
Note that this implies that $S_{\Gamma_{n,k}}\left(id_{G_{k}}\right)$
is non-amenable. Indeed, otherwise the free product $\bf A$ can be written as a union of finitely many 
right cosets of an amenable subgroup, which contradicts the fact that $\bf A$ is non-amenable. 
For each $x\in\mathsf{T}_{k}$, write $\theta_{x}$ for the projection
of $\oplus_{\mathsf{T}_{k}}{\bf A}$ to the $x$-coordinate, that is,
$\theta_x(\phi)=\phi(x).$ Then the reasoning above shows 
that the projection of $\Gamma_{n,k}\cap\left(\oplus_{\mathsf{T}_{k}}[{\bf A},{\bf A}]\right)$ under 
$\theta_{2^k}$ is non-amenable. Recall that $[{\bf A},{\bf A}]$ is a free group, 
By Lemma \ref{transitive}, the action of $\pi_{k}\left(\Gamma_{n,k}\right)$
is transitive on $\mathsf{T}_{k}$. It follows then 
$\theta_x\left(\Gamma_{n,k}\cap\left(\oplus_{\mathsf{T}_{k}}[{\bf A},{\bf A}]\right)\right)$
is a free group of rank at least $2$ for every vertex in $\mathsf{T}_{k}$.

Let $N$ be a normal subgroup of $\Gamma_{n,k}$, $N\neq\{id\}$. We need to
show that $N$ is non-amenable. Note that for each $x\in\mathsf{T}_{k}$,
$\theta_{x}\left(N\cap\oplus_{\mathsf{T}_{k}}[{\bf A},{\bf A}]\right)$
a normal subgroup of $\theta_{x}\left(\Gamma_{n,k}\cap\left(\oplus_{\mathsf{T}_{k}}[{\bf A},{\bf A}]\right)\right)$,
while the latter is a free group of rank at least $2$. Thus if on
the contrary $N$ is amenable, then $N\cap\oplus_{\mathsf{T}_{k}}[{\bf A},{\bf A}]=\{id\}$.
Since ${\bf A}/[{\bf A},{\bf A}]$ is finite, such trivial intersection
implies that $S_{N}(id_{G_{k}})$ is finite. 

We now argue that $S_{N}(id_{G_{k}})$ being a finite group contradicts
with the condition that $N$ is a non-trivial normal subgroup of $\Gamma_{n,k}$.
Note that since $G_{k}$ is finite, $S_{N}(id_{G_{k}})$ being finite
implies that $N$ is finite. On the other hand, since for each $x\in\mathsf{T}_{k}$,
$\theta_{x}\left(\Gamma_{n,k}\cap\left(\oplus_{\mathsf{T}_{k}}[{\bf A},{\bf A}]\right)\right)$
is a free group of rank at least $2$, it follows that for any element
$h\in\Gamma_{n,k}$, $h\neq id$, its conjugacy class is infinite.
Therefore $S_{N}(id_{G_{k}})$ being a finite group implies that $N=\{id\}$.
We conclude that a nontrivial normal subgroup of $\Gamma_{n,k}$ is
non-amenable. 

Since $[{\bf A},{\bf A}]$ is a free group, the only amenable subgroups
are the trivial group and the cyclic groups. It follows that the direct
sum $\oplus_{\mathsf{T}_{n}}[{\bf A},{\bf A}]$ has only countably
many amenable subgroups. The property of having only countably many
amenable subgroups is clearly preserved under taking finite extensions
and taking subgroups. Thus by Lemma \ref{virtual}, $\Gamma_{n,k}$
has only countably many amenable subgroups. 

\end{proof}

To bound the asymptotic entropy from above, we simply use the well-known
``fundamental inequality'', see e.g.\ \cite{BHM2008}. More precisely,
let $\mu$ be a probability measure on $\mathbf{F}_{d}$ with finite
first moment, $\pi:\mathbf{F}_{d}\to\Gamma$ an epimorphism. Let $S=\pi({\bf S})$
be the induced marking on $\Gamma$ and $\bar{\mu}=\pi\circ\mu$ be
the pushforward of $\mu$. The fundamental inequality implies that
\[
h(\Gamma,\bar{\mu})\le v_{\Gamma,S}\cdot\ell_{\Gamma,\bar{\mu}},
\]
where $v_{\Gamma,S}$ and $\ell_{\Gamma,\bar{\mu}}$ are asymptotic
volume growth rate and asymptotic speed with respect to generating
set $S$:
\[
v_{\Gamma,S}=\lim_{r\to\infty}\frac{1}{r}\log V_{\Gamma,S}(r)\mbox{ and }\ell_{\Gamma,\bar{\mu}}=\lim_{n\to\infty}\frac{1}{n}\sum_{g\in\Gamma}|g|_{S}\bar{\mu}^{(n)}(g).
\]
By sub-additivity, we have $\ell_{\Gamma,\bar{\mu}}\le\sum_{g\in\Gamma}|g|_{S}\bar{\mu}(g)\le\sum_{g\in\mathbf{F}_{d}}|g|_{\mathbf{S}}\mu(g).$
Thus the asymptotic entropy can be bounded by 
\begin{equation}
h\left(\Gamma,\bar{\mu}\right)\le v_{\Gamma,S}\sum_{g\in\mathbf{F}_{d}}|g|_{\mathbf{S}}\mu(g).\label{eq:fund}
\end{equation}
The estimate (\ref{eq:fund}) is the only place where the moment condition
on $\mu$ is needed. 

By Lemma \ref{conv}, the ball of radius $2^{k-2}$ around $id$ in
the Cayley graph of $(W_{k},(a_{k},b_{k}))$ coincide with the ball
of radius $2^{k-2}$ around $id$ in $(G,(a,b))$. It follows by sub-multiplicity
of the volume growth function that the asymptotic volume rate satisfy
\[
v_{W_{k},(a_{k},b_{k})}\le\frac{1}{2^{k-2}}\log V_{G,(a,b)}\left(2^{k-2}\right).
\]
Recall the maximal length $\ell_{n}$ defined in (\ref{eq:ln}). By
comparing lengths of generators, we have that for the subgroup $\Gamma_{n,k}$
of $W_{k}$ satisfies 
\[
v_{\Gamma_{n,k},S_{n,k}}\le\ell_{n}v_{W_{k},(a_{k},b_{k})}.
\]

\subsection{Proof of Proposition \ref{sequence}}

We are now ready to prove Proposition \ref{sequence} stated in the
Introduction.

\begin{proof}[Completion of proof of Proposition \ref{sequence}]

Let $\mu$ be a non-degenerate probability measure on ${\bf F}_{d}$
given, where $d\ge2$ and $\mu$ is of finite first moment. For each
$n\in\mathbb{N}$, choose $k_{n}\gg\log_{2}(n\ell_n)$ such that $\Gamma_{n,k_{n}}$
is defined as in subsection \ref{subsec:Choices-of-marked} and moreover
\[
\frac{\ell_{n}}{2^{k_{n}-2}}\log V_{G,(a,b)}\left(2^{k_{n}-2}\right)\le\frac{1}{n}.
\]
This is possible because the Fabrykowski-Gupta group $G$ has sub-exponential
volume growth, that is, 
\[
v_{G,(a,b)}=\lim_{r\to\infty}\frac{1}{r}\log V_{G,(a,b)}(r)=0,
\]
 Now we verify that the sequence $\left(\Gamma_{n,k_{n}},S_{n,k_{n}}\right)$
satisfy the properties stated. 
\begin{description}
\item [{(i)}] By construction, the generating tuple $S_{n,k_{n}}$ do not
satisfy any reduced word $w$ of length $|w|\in\left[1,n\right]$. 
\item [{(ii)}] The fundamental inequality (\ref{eq:fund}) implies that
with respect to the marking $({\bf F}_{d},{\bf S})\to(\Gamma_{n,k_{n}},S_{n,k_{n}})$,
\[
h\left(\Gamma_{n,k_{n}},\mu\right)\le v_{\Gamma_{n,k_{n}},S_{n,k_{n}}}\sum_{g\in\mathbf{F}_{d}}|g|_{{\bf S}}\mu(g)\le\frac{1}{n}\sum_{g\in\mathbf{F}_{d}}|g|_{{\bf S}}\mu(g).
\]
Thus the sequence of asymptotic entropies converge to $0$ as $n\to\infty$.
\item [{(iii)}] This property is shown in Lemma \ref{free}.
\end{description}
The proof of Proposition \ref{sequence} is complete.

\end{proof}

\begin{remark}

For $d\ge3$, in the proof of Proposition \ref{sequence} one can
use the first Grigorchuk group $G_{012}=\left\langle a,b,c\right\rangle $
introduced in \cite{grigorchuk1980,grigorchuk1984} instead of the
Fabrykowski-Gupta group. Recall that $G_{012}$ acts on the rooted
binary tree. Then one can consider the permutational wreath extension
${\bf B}\wr_{\mathsf{T}_{n}}G_{n}$, where $G_{n}=G/{\rm St}_{G}(n)$
and ${\bf B}=(\mathbb{Z}/2\mathbb{Z})\ast\left(\mathbb{Z}/2\mathbb{Z}\times\mathbb{Z}/2\mathbb{Z}\right)=\left(\left\langle {\bf s}\right\rangle \right)\ast\left(\left\langle {\bf t}\right\rangle \times\left\langle {\bf u}\right\rangle \right)$.
Similar to the sequence of extensions $\Gamma_{n}$, set 
\[
H_{n}=\left\langle a_{n},b_{n},c_{n}\right\rangle <{\bf B}\wr_{\mathsf{T}_{3n}}G_{3n}
\]
where the generators are defined as 
\[
a_{n}=(id,\bar{a}),\ b_{n}=\left(\delta_{1^{3n}}^{{\bf t}}+\delta_{1^{3n-1}0}^{{\bf s}},\bar{b}\right),\ c_{n}=\left(\delta_{1^{3n}}^{{\bf u}}+\delta_{1^{3n-1}0}^{{\bf s}},\bar{c}\right).
\]
Similar proof as in this section with $\Gamma_{n}$ replaced by $H_{n}$
shows that for $d\ge3$, the statement Proposition \ref{sequence}
is true under the weaker assumption that $\mu$ has finite $\alpha_{0}$-moment
and finite entropy, where $\alpha_{0}$ is the exponent in the growth
upper bound $v_{G_{012}}(r)\lesssim e^{r^{\alpha_{0}}}$ from \cite{bartholdi1998,muchnikpak},
$\alpha_{0}\approx0.7674$. 

We choose to take extensions of the Fabrykowski-Gupta group $G$ here
because the resulting groups are $2$-generated, which allows to cover
the case $d=2$. It is remarked in \cite{francoeurgarrido} that all
maximal subgroups of $G$ are of finite index, which would imply that
there is a sequence of marking $S_{k}$ on $G$ such that $(G,S_{k})$
converges to the free group $\left(\mathbf{F}_{d},\mathbf{S}\right)$
when $k\to\infty$ by \cite{bartholdi2015order}. Since we could not
find a written proof of this statement, in this section we produce
tuples of elements of $G$, which a priori do not necessarily generate
$G$, where only Ab\'ert's criterion \cite{abert2005} is invoked. 

\end{remark}

\section{stationary joinings and proof of the main results\label{sec:joining}}

Let $(X,\nu)$ and $(Y,\eta)$ be two $\mu$-stationary $G$-spaces.
Following \cite{furstenbergglasner2010stationary}, we say a probability
measure $\lambda$ on $X\times Y$ is a stationary joining of $\nu$
and $\eta$ if it is $\mu$-stationary and its marginals are $\nu$
and $\eta$ respectively. 

In this section we focus on the situation where both stationary systems
are $(G,\mu)$-boundaries. We use notations introduced in Section
\ref{subsec:boundaries}. Denote by $(B,\nu_{B})$ a compact model
of the Poisson boundary of $(G,\mu)$. Let $(X,\nu)$ and $(Y,\eta)$
be compact models of two $(G,\mu)$-boundaries and denote by $\boldsymbol{\beta}_{X}$
and $\boldsymbol{\beta}_{Y}$ the corresponding maps from the Poisson
boundary $(B,\nu_{B})$ to $(X,\nu)$ and $(Y,\eta)$. Consider the
map 
\begin{align*}
\boldsymbol{\beta}_{X}\times\boldsymbol{\beta}_{Y}:B & \to X\times Y\\
b & \mapsto\left(\boldsymbol{\beta}_{X}(b),\boldsymbol{\beta}_{Y}(b)\right),
\end{align*}
and denote by $Z$ the range $\left(\boldsymbol{\beta}_{X}\times\boldsymbol{\beta}_{Y}\right)(B)$
and $\nu\varcurlyvee\eta$ the pushforward of the harmonic measure
$\nu_{B}$ under $\boldsymbol{\beta}_{X}\times\boldsymbol{\beta}_{Y}$.
Then it's clear by definition that $(Z,\nu\varcurlyvee\eta)$ is
a $G$-factor of the Poisson boundary $\left(B,\nu_{B}\right)$, in
other words, it is a $(G,\mu)$-boundary. The $G$-space $(Z,\nu\varcurlyvee\eta)$
is the unique stationary joining of the $\mu$-boundaries $(X,\nu)$
and $(Y,\eta)$, see \cite[Proposition 3.1]{furstenbergglasner2010stationary}.

On the level of groups, given two $d$-marked groups $(G_{1},S_{1})$
and $(G_{2},S_{2})$, one can take their diagonal product, denoted
by $\left(G_{1}\otimes G_{2},S\right)$, as the subgroup of $G_{1}\times G_{2}$
generated by 
\[
S=\left(\left(s_{1}^{(1)},s_{1}^{(2)}\right),\ldots\left(s_{d}^{(1)},s_{d}^{(2)}\right)\right),
\]
where $S_{i}=\left(s_{1}^{(i)},\ldots,s_{d}^{(i)}\right)$, $i=1,2$.
This operation on two groups corresponds to taking stationary joinings
of the Poisson boundaries:

\begin{lemma}

Let $\mu$ be a probability measure on ${\bf F}_{d}$. The Poisson
boundary of $\left(G_{1}\otimes G_{2},\mu\right)$ is the stationary
joining of the Poisson boundaries of $(G_{1},\mu)$ and $(G_{2},\mu)$. 

\end{lemma}

\begin{proof}

Denote by $(B_{i},\nu_{i})$ the Poisson boundary of $(G_{i},\mu)$,
$i=1,2$ and regard them as $G_{1}\otimes G_{2}$-spaces. Denote by
$(Z,\nu_{1}\varcurlyvee\nu_{2})$ the stationary joining of $(B_{1},\nu_{1})$
and $(B_{2},\nu_{2})$ as above and $\pi_{i}:Z\to B_{i}$ the projections.
We need to show $(Z,\nu_{1}\varcurlyvee\nu_{2})$ is the maximal
$(G_{1}\otimes G_{2},\mu)$-boundary. 

Let $(Y,\eta)$ be a $(G_{1}\otimes G_{2},\mu)$-boundary. Denote
by $K_{i}$ the subgroup of $G_{1}\otimes G_{2}$ which consists of
elements that project to identity in $G_{i}$, that is,
\[
K_{i}=\left\{ (g_{1},g_{2})\in G_{1}\times G_{2}:\ (g_{1},g_{2})\in G_{1}\otimes G_{2},\ g_{i}=id_{G_{i}}\right\} .
\]
Denote by $Y_{i}=Y//K_{i}$ the space of $K_{i}$-ergodic components
of $Y$ and $\eta_{i}$ the pushforward of the measure $\eta$ under
the $K_{i}$-factor map $Y\to Y//K_{i}$. Since $(Y,\eta)$ is an
ergodic $G_{1}\otimes G_{2}$-space and $K_{1}\cap K_{2}=\{id\}$,
we have that $Y$ can be viewed as a subset of $Y_{2}\times Y_{1}$.
It's easy to see that by definition of $K_{i}$ that $G_{1}\otimes G_{2}/K_{2}\simeq G_{1}$.
It follows that $(Y_{2},\eta_{2})$ is a $(G_{1},\mu)$-boundary.
Denote by ${\bf \beta}_{Y_{2}}$ the boundary map from $(B_{1},\nu_{1})$
to $(Y_{2},\eta_{2})$. In the same way we have $(Y_{1},\eta_{1})$
is a $(G_{2},\mu)$-boundary and denote by ${\bf \beta}_{Y_{1}}:(B_{2},\nu_{2})\to(Y_{1},\eta_{1})$
the boundary map. By uniqueness of stationary joinings of $\mu$-boundaries,
we have that $(Y,\eta)=(Y_{2}\times Y_{1},\eta_{2}\varcurlyvee\eta_{1})$.
It follows that $(Y,\eta)$ is a factor of $(Z,\nu_{1}\varcurlyvee\nu_{2})$,
where the boundary map is given by $z\mapsto\left({\bf \beta}_{Y_{2}}\circ\pi_{1}(z),{\bf \beta}_{Y_{1}}\circ\pi_{2}(z)\right)$. 

\end{proof}

With the sequence of marked groups provided by Proposition \ref{sequence},
we are now ready to complete the proofs of Theorem \ref{main} and
\ref{B}.

\begin{proof}[Proof of Theorem \ref{main}]

Denote by $(B,\nu_{B})$ the Poisson boundary of $({\bf F}_{d},\mu)$.
Let $\left((\Gamma_{k},S_{k})\right)_{k=1}^{\infty}$ be a sequence
marked groups provided by Proposition \ref{sequence}. Denote by $(\Pi_{k},\eta_{k})$
the Poisson boundary of $(\Gamma_{k},\mu)$. Since $(\Gamma_{k},S_{k})$
can be identified with a projection $\pi_{k}:{\bf F}_{d}\to\Gamma_{k}$,
we regard $(\Pi_{k},\eta_{k})$ as a $({\bf F}_{d},\mu)$-space, where
the ${\bf F}_{d}$-action factors through $\pi_{k}$.

Since $\Gamma$ is a proper quotient of ${\bf F}_{d}$, $N=\ker(\pi:{\bf F}_{d}\to\Gamma)$
is nontrivial. Fix a choice of element $g\in N$, $g\neq id$. Choose
an index $k\in\mathbb{N}$ sufficiently large such that the balls
of radius $2|g|_{{\bf S}}$ around identities in $(\Gamma_{k},S_{k})$
and $({\bf F}_{d},{\bf S})$ coincide and $h(\Gamma_{k},\mu)<\epsilon$.
Take $\tilde{\Gamma}$ to be the diagonal product $\left(\Gamma\otimes\Gamma_{k},S\right)$.
Then 
\[
h(\Gamma\otimes\Gamma_{k},\mu)\le h(\Gamma,\mu)+h(\Gamma_{k},\mu)<h(\Gamma,\mu)+\epsilon.
\]
Since $g$ acts trivially on the Poisson boundary of $(\Gamma,\mu)$
but acts freely on $(\Pi_{k},\nu_{k})$, it follows that $(\Pi_{k},\nu_{k})$
is not a ${\bf F}_{d}$-factor of the Poisson boundary of $(\Gamma,\mu)$.
By Lemma \ref{mp}, we conclude that $h(\Gamma\otimes\Gamma_{k},\mu)>h(\Gamma,\mu)$. 

\end{proof}

\begin{proof}[Proof of Theorem \ref{B}]

The proof is similar to Theorem \ref{main}. Since $(X,\nu)$ is assumed
to be a $({\bf F}_{d},\mu)$-boundary where the action of ${\bf F}_{d}$
is not essentially free, we can choose an element $g\in{\bf F}_{d}$,
$g\neq1$, such that $\nu({\rm Fix}_{X}(g))>0$. Choose an index $k\in\mathbb{N}$
sufficiently large such that the balls of radius $2|g|_{{\bf S}}$
around identities in $(\Gamma_{k},S_{k})$ and $({\bf F}_{d},{\bf S})$
coincide and $h(\Gamma_{k},\mu)<\epsilon$. Take the stationary
joining $\left(Z_{k},\nu\varcurlyvee\eta_{k}\right)$ of $(X,\nu)$
and $(\Pi_{k},\eta_{k})$. By the general inequality, we have 
\[
h(Z_{k},\nu\varcurlyvee\eta_{k})\le h(X,\nu)+h(\Pi_{k},\eta_{k})\le h(X,\nu)+\epsilon.
\]
It remains to show that $h(Z_{k},\nu\varcurlyvee\eta_{k})>h(X,\nu)$.
Suppose on the contrary equality holds, then by Lemma \ref{mp}, the
equality would imply $\left(Z_{k},\nu\varcurlyvee\eta_{k}\right)=(X,\nu)$.
However the action of $\Gamma_{k}$ on $(\Pi_{k},\eta_{k})$ is essentially
free, which implies $\nu\varcurlyvee\eta_{k}\left({\rm Fix}_{Z_{k}}(g)\right)=0$,
contradicting $\nu({\rm Fix}_{X}(g))>0$.

\end{proof}

We now show an analogous result on spectral radii stated as Theorem
\ref{spec} in the Introduction. Consider a symmetric non-degenerate
probability measure $\mu$ on $\Gamma$. In \cite{kesten1959full,kesten1959symmetric}
Kesten proved the following theorem: let $\mu$ be a symmetric non-degenerate
probability measure on $\Gamma$ and $N$ be a normal subgroup of
$\Gamma$, then the following are equivalent:
\begin{description}
\item [{(i)}] $\rho(\Gamma,\mu)=\rho(\Gamma/N,\mu)$,
\item [{(ii)}] $N$ is amenable.
\end{description}
Given a proper quotient $\Gamma$ of ${\bf F}_{d}$ and $\epsilon>0$,
to prove Theorem \ref{spec} we take $\tilde{\Gamma}$ to be a diagonal
product $\Gamma\otimes H$, for some appropriate choice of $H$ similar
to the groups used in Theorem \ref{main}. 

\begin{proof}[Proof of Theorem \ref{spec}]

Let $\epsilon>0$ be a constant given. Let $\Gamma$ be a proper quotient
of ${\bf F}_{d}$ and fix a choice of $g_0\in\ker({\bf F}_{d}\to\Gamma)$,
$g_0\neq id$. Take $n\ge 2|g_0|_{{\bf S}}$. As in the beginning of Subsection \ref{subsec:Choices-of-marked},
fix a choice of $d$-tuple of elements in the Fabrykowski-Gupta
group $G$, $T=\left(a\gamma_{1}^{(n)},b\gamma_{2}^{(n)},\gamma_{3}^{(n)},\ldots,\gamma_{d}^{(n)}\right)$
that do not satisfy any reduced word of length at most $n$. 
Denote the group generated by this tuple by $G_n$.
Take first the diagonal product $\Gamma\otimes G_n$. By the choice of $g$
and marking on $G$ we have that $N_{0}=\ker(\Gamma\otimes G_n\to\Gamma)$
is non-trivial. Note that $N_0$ can be regarded as a normal subgroup of $G_n$. 

Denote by $(W_{\ell})$ a $\mu$-random walk on ${\bf F}_{d}$. For a
marked group $(H,S)$, we write $\pi_{H}$ for the quotient map ${\bf F}_{d}\to H$
when the marking is clear from the context.

Take a small constant $\epsilon_{1}>0$, choose $\ell$ large enough
such that 
\[
\mathbb{P}\left(\pi_{\Gamma}\left(W_{\ell}\right)=id_{\Gamma}\right)\ge((1-\epsilon_{1})\rho(\Gamma,\mu))^{\ell}.
\]
For $g\in N_{0}$, set 
\[
Q(g)=\frac{\mathbb{P}(\pi_{\Gamma\otimes G}(W_{\ell})=g)}{\mathbb{P}(\pi_{\Gamma}(W_{\ell})=id_{\Gamma})}.
\]
Then $Q$ is a symmetric probability measure on $N_{0}$. Equip $N_{0}$
with the induced metric $|\cdot|_{T}$ from $(G_n,T)$. Let $R$ be
a sufficiently large radius such that $Q(\{\gamma\in N_{0}:|\gamma|_{T}>R\})\le\epsilon_{1}$.
Truncate the measure $Q$ at $R$ and let 
\[
Q_{R}(g)=\frac{1}{Q\left(\{\gamma:|\gamma|_{T}\le R\}\right)}Q(g){\bf 1}_{\{|g|_{T}\le R\}}.
\]
Since $N_{0}$ is a subgroup of the Fabrykowski-Gupta group $G$, thus amenable, there
exists an integer $m$ such that 
\[
Q_{R}^{2m}\left(id_{N_{0}}\right)\ge(1-\epsilon_{1})^{2m}.
\]

With $\ell,m,R$ chosen as above, for a sufficiently large index $k$, to
be specified shortly, and take the marked group $\left(\Gamma_{n,k}, S_{n,k}\right)$
defined in Subsection \ref{subsec:Choices-of-marked}. 
Consider the diagonal product $\Gamma\otimes\Gamma_{n,k}$. 
By Lemma \ref{conv}, the
ball of radius $2^{k-2}$ around $id$ in the Cayley graph of $(W_{k},(a_{k},b_{k}))$
coincide with the ball of radius $2^{k-2}$ around $id$ in $(G,(a,b))$. Choose $k$ sufficiently large such that 
$k>2mR$.

Now we follow the original argument in Kesten's theorem $(ii)\Rightarrow(i)$
above to show $\rho(\Gamma\otimes\Gamma_{n,k},\mu)>\rho(\Gamma,\mu)-\epsilon$.
Write $W_{(k-1)\ell}^{k\ell}=W_{(k-1)\ell}^{-1}W_{k\ell}$.
\begin{multline*}
\mathbb{P}\left(\pi_{\Gamma\otimes\Gamma_{n,k}}\left(W_{2\ell m}\right)=id_{\Gamma\otimes\Gamma_{n,k}}\right)\\
\ge\mathbb{P}\left(\cap_{k=1}^{2m}\left\{
\pi_{\Gamma}\left(W_{(k-1)\ell}^{k\ell}\right)=id_{\Gamma},\left|\pi_{\Gamma_{n,k}}\left(W_{(k-1)\ell}^{k\ell}\right)\right|_{S_{n,k}}\le
R\right\}
 \cap\left\{ \pi_{\Gamma_{\ell}}\left(W_{2\ell m}\right)=id_{\Gamma_{n,k}}\right\} \right)\\
\ge((1-\epsilon_{1})\rho(\Gamma,\mu))^{2\ell m}(1-\epsilon_{1})^{2m}Q_{R}^{2m}\left(id_{N_{0}}\right)\ge(1-\epsilon_{1})^{2mn+4m}\rho^{2mn}.
\end{multline*}
Choose $\epsilon_{1}<\epsilon/3$, we have that $\rho(\Gamma\otimes\Gamma_{n,k},\mu)>(1-\epsilon)\rho(\Gamma,\mu)$. 

Finally, by Lemma \ref{free}, $\Gamma_{n,k}$ has
no nontrivial amenable normal subgroups. Since by the choice of markings
$\ker(\Gamma\otimes G_{n,k}\to\Gamma)$ is nontrivial, 
the kernel is non-amenable. By Kesten's theorem $(i)\Rightarrow(ii)$, we conclude
that $\rho(\tilde{\Gamma},\mu)<\rho(\Gamma,\mu)$. 

\end{proof}

\bibliographystyle{plain}
\bibliography{refs}

\end{document}